\documentclass[a4paper,11pt,reqno]{amsart}
\usepackage{graphicx}
\usepackage{a4wide}
\usepackage{amssymb}
\usepackage{amsthm}
\usepackage{eucal}
\usepackage[pdftex,colorlinks]{hyperref}

\numberwithin{equation}{section}

\newcommand{\be}{\begin{equation}}
\newcommand{\ee}{\end{equation}}
\newcommand{\ba}{\begin{eqnarray}}
\newcommand{\ea}{\end{eqnarray}}

\newtheorem{theorem}{Theorem}[section]
\newtheorem{proposition}[theorem]{Proposition}
\newtheorem{remark}[theorem]{Remark}
\newtheorem{question}[theorem]{Question}
\newtheorem{lemma}[theorem]{Lemma}

\begin{document}

\title [The cost of null controllability for the Stokes system]{A hyperbolic system and the cost of the null controllability for the Stokes system} 

\author[F. W. Chaves-Silva]{F. W. Chaves-Silva}
\address{BCAM -- Basque Center for Applied Mathematics
Mazarredo 14, 48009 Bilbao, Basque Country, Spain}
\email{chaves@bcamath.org}



\maketitle
\begin{abstract} 
This paper is devoted  to study  the cost of the null controllability for the Stokes system. Using the control transmutation method we show that the cost of driving the Stokes system to rest at time $T$ is of order $e^{C/T}$ when $T \longrightarrow 0^+$, i.e., the same order of controllability as for the heat equation. For this to be possible,  we are led  to prove a new exact controllability result for a  hyperbolic system with a resistance term, which will be done under assumptions on the control region.
\end{abstract}


\section{Introduction}
Let  $\Omega \subset \mathbb{R}^N$ $(N \geq 2)$ be a bounded connected open set, whose boundary $\partial \Omega$ is regular enough. Let $T > 0$  and let  $\omega$  be a nonempty subset of $\Omega$ which will usually be referred to as a \textit{control domain}. We will use the notation $Q :=  \Omega \times (0,T) $ and  $\Sigma := \partial \Omega \times (0,T)$ and we will denote by $\nu(x)$ the outward normal to $\Omega$ at the point $x \in \partial \Omega$.

Given $u_0 \in L^2(\Omega)$, it is well-known (see  \cite{FC-Z, F-I1}) that  there exists $f \in L^2(\omega \times (0,T))$ such that the associated solution $v$ to the heat equation

\begin{equation}\label{heat}
\left |   
\begin{array}{ll}
v_t  - \Delta v   = f1_\omega &  \mbox{in}  \    Q,  \\
v = 0 & \mbox{on} \  \Sigma, \\
v(0) = v_0 & \mbox{in} \ \Omega
\end{array}
\right. 
\end{equation}
satisfies:
\be
v(T)=0.
\ee
In other words, the heat equation is \textit{null controllable} for any control domain and any initial data $v_0 \in L^2(\Omega)$. Moreover, one also has the following estimate:
\be\label{costheat}
||f1_\omega||_{L^2(Q)} \leq C_h||v_0||_{L^2(\Omega)},
\ee
for a constant $C_h$, the \textit{cost of controllability for the heat equation},  of the form  $e^{C(\Omega, \omega)(1+1/T)}$, i.e., the heat equation has a cost of controllability of order $e^{C/T}$ as $T \longrightarrow 0^+$.

As pointed out in \cite{EZ} (see also \cite{EZ-1, Miller, Miller-1, TT}), the main reason for the form of the constant $C_h$ in \eqref{costheat} is due to the fact that the fundamental solution of the heat equation in $\mathbb{R}^N$ is given by 
 \be\label{funda}
 \Phi(x,t) = \frac{1}{(4\pi t)^{N/2}}e^{-\frac{|x|^2}{4t}}.
 \ee

As in the case of the heat equation, if one now considers the Stokes system
\begin{equation}\label{stokes}
\left |   
\begin{array}{ll}
y_t - \Delta y  +\nabla p = g1_\omega &  \mbox{in}  \   Q,  \\
div \ y  = 0 &  \mbox{in}  \    Q,  \\
y = 0 & \mbox{on} \ \Sigma, \\
y(0) = y_0 & \mbox{in} \ \Omega,
\end{array}
\right. 
\end{equation}
it is also well-known (see, for instance, \cite{FC-G-P}) that, given $y_0 \in L^2(\Omega)$ with  $ div \ y_0 =0$,   there exists $g \in L^2(\omega \times (0,T))$ such that the associated solution $y_0$ of \eqref{stokes} satisfies:
$$
y(T) = 0.
$$
Nevertheless, unlike the case of the heat equation, for the Stokes system, the known results in the literature (see, for instance, \cite{FC-G-P}) give 
\be\label{coststokes}
||g1_\omega||_{L^2(Q)} \leq C_S||y_0||_{L^2(\Omega)},
\ee
for a constant $C_S$, the \textit{cost of controllability for the Stokes system}, of the form  $e^{C(\Omega, \omega)(1+1/T^4)}$, i.e., the Stokes system has a cost of controllability   at most of order $e^{C/T^4}$ as $T \longrightarrow 0^+$.

Since the fundamental solutions of the heat equation and the Stokes system have, at least for $N=2, 3$, the same behavior in time (see \cite{Guen, Guen1, Solo}), looking to \eqref{costheat} and \eqref{coststokes},  the following natural question arises:

\begin{question}\label{q1}
Do the costs of the controllability for the heat equation and the Stokes system have the same order in time  as $T \longrightarrow 0^+$? 
\end{question}

When trying to answer Question \ref{q1}, the first attempt is to analyze  the different ways one can prove \eqref{costheat} and \eqref{coststokes}. In fact, there are at least two different ways one can prove \eqref{costheat}. The first one is based on spectral decompositions, the so-called Lebeau-Robbiano strategy (see \cite{LR}), the  second one is  based on the use of Carleman inequalities (see  \cite{FC-Z, F-I1}). For the Stokes system,  since one must deal with the pressure, it seems that a Lebeau-Robbiano strategy  is very difficult to prove and, as far as we know, it has not been proved yet to hold. Consequently,  the most known method used to prove \eqref{coststokes} is based on Carleman inequalities (see \cite{FC-G-P}).  

The main difference when proving \eqref{costheat} and \eqref{coststokes} by means of Carleman inequalities are the weights one must use. Indeed, for the heat equation the weights used are of the form 
\begin{equation}\label{weight-heat-90}
\rho(t) = \frac{e^{C/(t(T-t))}}{t(T-t)},
\end{equation}
while for the Stokes system the weights take the form 
\begin{equation}\label{weight-stokes-90}
\rho(t) = \frac{e^{C/(t^4(T-t)^4)}}{t^4(T-t)^4}. 
\end{equation}

If we were able to use weights as \eqref{weight-heat-90} for the Stokes system  then the two equations would have costs of controllability of same order. However, a careful analysis in both proofs indicates that the obstruction to have weights of the form \eqref{weight-heat-90} for the Stokes system is due to the pressure term and that, probably, it is of purely technical nature.

The main objective of this paper is to show that, at least for good geometries, the heat and the Stokes system have costs of controllability of same order as the time goes to zero. Our strategy will  not be based on the use of Carleman inequalities but rather on the application of the Control Transmutation Method (CTM). 

In order to use the CTM,  we are led to study the null controllability of the  following hyperbolic system with a pressure term:
\begin{equation}\label{hyperpress}
\left |   
\begin{array}{ll}
u_{tt}  - \Delta u + \nabla p = h1_{\omega} &  \mbox{in}  \   Q,  \\
div \ u  = 0 & \mbox{in} \  Q, \\
u = 0 & \mbox{on} \  \Sigma, \\
u(0) = u^0,  u_t(0) = u^1 & \mbox{in} \  \Omega.
\end{array}
\right. 
\end{equation}

The idea  is as follows. If one can show that system \eqref{hyperpress} is null controllable, then the CTM can be applied to guarantee the null controllability of the Stokes system \eqref{stokes}. Moreover, if  the cost of  controlling \eqref{hyperpress} is known,  then the cost of the controllability for  \eqref{stokes} can be obtained explicitly (see \cite{Miller}).

It is important to mention that systems like  \eqref{hyperpress} are  simple models of dynamical elasticity for incompressible materials. They also appear in coupled elasto-thermicity  problems where one of the coupling parameters (related to compressibility properties) tends to infinity (see \cite{Lions-1}).

Concerning the controllability of \eqref{hyperpress}, as far as we know, the only result available in the literature is \cite{Rocha}. There, the author shows the exact controllability of \eqref{hyperpress} when the control is acting on a part of the boundary. However, it seems that no controllability results are known when the control is acting internally, i.e., acting on a part of the domain. The main reason for that seems to be the fact that system \eqref{hyperpress} is not of Cauchy-Kowalewski type, which makes impossible to  use directly  Holgrem's Theorem as in the case of  the wave equation.

This paper is organized as follows.  In section \ref{sec2},  we prove  that, for regular initial data, system \eqref{stokes} has the a cost of controllability of the same order of  the one for the heat equation. In section \ref{sec4}, we consider less regular data and improve the results of section \ref{sec2}. Section \ref{prooftheorem2.2} is devoted to prove the internal null controllability of system \eqref{hyperpress}.

\section{The Stokes system with regular initial data}\label{sec2}

In this section, we prove that if the initial data is regular enough then the Stokes system \eqref{stokes} is null controllable with a cost of order $e^{C/T}$ as $T \longrightarrow 0^+$.  Our proof is based on the Control Transmutation Method in the spirit of \cite{Miller}  and a null controllability result for system \eqref{hyperpress}.

We  assume that $\Omega$ is star-shaped with respect to the origin, i.e., there exists $ \gamma > 0 $ such that
$$
x \cdot \nu (x) \geq \gamma > 0,  \ \  \forall x \in \partial \Omega.
$$ 
and we define
\begin{align}
R_0 = \displaystyle \max\{|m(x)|, \ x \in \bar{\Omega}\}.
\end{align}

Our control region  $\omega$ will be   a nonempty subset of $ \Omega$  satisfying
\begin{equation}\label{hipm}
 \exists \mathcal{O} \subset \mathbb{R}^N, \mathcal{O} \ \mbox{being a neighborhood of } \  \partial \Omega \ \mbox{and}  \ \omega = \Omega \cap \mathcal{O}.
 \end{equation}

We also define the following usual spaces in  the context of fluid mechanics:

$$
\mathcal{V} = \{  v \in C^{\infty}_0(\Omega); \ div \ v =0 \},
$$
$$
V =  \overline{\mathcal{V}}^{H^1_0(\Omega)^N} =  \{ u \in H^1_0(\Omega)^N; \ div \ u = 0\},
$$
$$
H = \overline{\mathcal{V}}^{L^2(\Omega)^N} = \{ u \in L^2(\Omega)^N; \ div \ u = 0, u \cdot \nu = 0 \ \mbox{on} \ \partial \Omega \}.
$$



The main result of this section is stated as follows.

\begin{theorem}\label{nulls}
Assume  $\omega$ satisfies \eqref{hipm}, $y_0 \in V$ and let  $0 < T\leq 1$. Then there exists a control $g \in L^2(\omega \times (0,T))$ such that the solution $y$ of \eqref{stokes} satisfies: 
$$
y(T)=0.
$$
Moreover, there exist positive constants $C_1$ and  $C_2$, depending only on $\Omega$ and $\omega$,  such that
\begin{equation}\label{costs}
\int \! \! \! \int_{\omega \times (0,T)}|g|^2dxdt \leq C_1e^{C_2/T}||y_0||^2_V.
\end{equation}
\end{theorem}

\begin{proof}[Proof of Theorem \ref{nulls}]

For the  proof of Theorem \ref{nulls}, we need the following results.

\begin{theorem}\label{nullw}
Assume  $\omega$ satisfies \eqref{hipm}. Then there exists $T_0 > 0$ such that,  for any  $T >  T_0$ and any $(u_0,u_1) \in V\times H$, we can find  a control $h \in L^2(0,T; H)$ such that the associated solution $u$ of \eqref{hyperpress} satisfies:
$$
u(T) =u_t(T) =0.
$$
Moreover, there exists $C>0$ such that 
\begin{equation}\label{costw}
\int \! \! \! \int_{\omega \times (0,T)}|h|^2dxdt  \leq C\bigl(||u_0||^2_V + |u^1|_{H}^2\bigl).
\end{equation}
\end{theorem}

\begin{lemma}\label{controlledfundamentalsolution}
There exists a positive constant $\alpha^*$  such that, for all $\alpha > \alpha^*$, there exists $\gamma > 0$ having the property that, for all $L > 0$ and $ T \in (0, \inf (\pi/2, L)^2]$, there exists a distribution $ k \in C([0,T]; \mathcal{M}(-L,L))$ satisfying
\begin{equation}\label{obs12}
\left |   
\begin{array}{ll}
k_{t}   =   \partial_{s}^2k   \ \  \mbox{in}  \  \ \mathcal{D'}((0,T)\times (-L,L)) ,  \\
k(0,x) = \delta(0),\\
k(T,x) = 0, \\
||k||^2_{L^2((0,T)\times (-L,L))}  \leq \gamma e^{\alpha L^2/T}.
\end{array}
\right. 
\end{equation}

\end{lemma}
 
We prove Theorem  \ref{nullw} in section \ref{prooftheorem2.2}. A proof of Lemma  \ref{controlledfundamentalsolution} can be found in \cite{Miller}.

Let us now  introduce two different time intervals $(0,T)$  and $(0,L)$ and consider the two systems
\begin{equation} \label{stokes-system}
\left |   
\begin{array}{ll}
y_t  - \Delta y + \nabla p = g1_{\omega} &   \mbox{in}  \  Q_t:= \Omega \times (0,T) ,  \\
div \ y  = 0 & \mbox{in} \ \ Q_t, \\
y = 0 & \mbox{on} \  \Sigma_t := \partial \Omega \times (0,T), \\
y(0) = y_0 & \mbox{in} \  \Omega
\end{array}
\right. 
\end{equation}
and 
\begin{equation} \label{hyperbolic-system}
\left |   
\begin{array}{ll}
u_{ll} - \Delta u + \nabla q = h1_{\omega}  &  \mbox{in}  \    Q_l :=\Omega \times (0,L),  \\
div \ u  = 0 & \mbox{in} \ Q_l, \\
u = 0 & \mbox{on}  \ \Sigma_l := \partial \Omega \times (0,L), \\
u(0) = y_0, \ u_l(0) = 0 & \mbox{in} \ \ \Omega
\end{array}
\right. 
\end{equation}
in $ \Omega \times (0,T)$ and $\Omega \times (0,L)$, respectively. Here, $l$ plays the role of a pseudo-timevariable.
\vskip0.2cm

Taking  $L > T_0$, where $T_0$ is the minimal time of Theorem \ref{nullw},  it follows from Theorem \ref{nullw}  that the system (\ref{hyperbolic-system}) is null controllable,  with a control $h \in L^2(\omega \times (0,L))$ satisfying \eqref{costw}.

\vskip0.2cm

Next, we extend $k$ by zero outside $[0,T]\times(-L,L)$,  $u$ and $h$ by reflection to $[-L,0]$ and by zero outside $[-L,L]$ and set
\be\label{ys}
y(t) = \int k(t,s)u(s)ds
\ee
and
\be\label{gs}
g(t) = \int k(t,s) h(s)ds.
\ee

From \eqref{obs12}, we see that
$$
y(0) = y_0 \ \text{and}  \ y(T) = 0
$$
and from \eqref{costw} and $\eqref{obs12}_4$, we have that 
$$
\int \! \! \! \int_{\omega \times (0,T)}|g|^2dxdt \leq C \gamma e^{\alpha L^2/T} ||y_0||^2_V.
$$

We finish the proof showing  that  the pair $(y,g)$ solves, together with some $p$, the Stokes system \eqref{stokes-system}.

First, it is not difficult to see that 
$$
div \ y = 0 \ \text{in} \ Q_t \ \text{and} \ y = 0 \ \text{on} \ \Sigma_t. 
$$

Now, for any $\varphi \in V$, we have

$$
<y(t),\varphi>_H = <\int k(t,s)u(s)ds,\varphi>_H,
$$
which implies 
$$
<y_t(t),\varphi>_H = <\int k_t(t,s)u(s)ds,\varphi>_H.
$$

Using  the properties of  $k$, we see that

$$
<y_t(t),\varphi>_H= <\int k_{ss}(t,s)u(s)ds,\varphi>_H.
$$

Integrating by parts, and using the fact that $u(-L) = u(L) = u_l(-L) = u_l(L) = 0$, we obtain

$$
<y_t(t),\varphi>_H = <\int k(t,s)u_{ss}(s)ds,\varphi>_Hds,
$$
i.e., 
$$
<y_t(t),\varphi>_H = \int k(t,s)<u_{ss}(s),\varphi>_Hds.
$$

Since $u$ is, together with some $q$,   solution of (\ref{hyperbolic-system}), we have
$$
<y_t(t),\varphi>_H= \int k(t,s)<\Delta u(s)+h1_{\omega},\varphi>_Hds.
$$

Therefore, 
$$
<y_t(t),\varphi>_{H} = <\int k(t,s)\Delta u(s)ds,\varphi>_H+ <\int k(t,s)h1_{\omega}ds,\varphi>_H.
$$

This last identity gives
\be\label{ws}
<y_t(t) - \Delta y(t),\varphi>_H=  <g(t)1_{\omega},\varphi>_H,
\ee
and the proof is finished.
\end{proof}



\section{The Stokes system with less regular data}\label{sec4}
In this section we improve the result obtained in section \ref{sec2}. Indeed, we prove that we can take less regular initial data and still have null controllability for the Stokes system with a cost of order $e^{C/T}$ as $T \longrightarrow 0^+$. In order to show the result, we combine Theorem \ref{nulls}, energy inequalities and the smoothing effect of the Stokes system. 

The result is as follows. 
\begin{theorem}\label{nulls-1}
Assume  $\omega$ satisfies \eqref{hipm},  $y_0 \in H$ and let $0 < T\leq 1$. Then there exists a control $g \in L^2(\omega \times (0,T))$ such that the solution $y$ of \eqref{stokes} satisfies: 
$$
y(T)=0.
$$
Moreover, there exist positive constants $C_1$ and  $C_2$, depending only on $\Omega$ and $\omega$, such that
\begin{equation}\label{costs}
\int \! \! \! \int_{\omega \times (0,T)}|g|^2dxdt \leq  C_1e^{C_2/T}|y_{0}|^2_H.
\end{equation}
\end{theorem}

\begin{proof}
%
%

We begin choosing $\epsilon >0$ small enough and letting system \eqref{stokes-system} evolve freely in the interval $(0,\epsilon)$. From the smoothing effect of the Stokes system, we have that   $y(\epsilon) = y_{\epsilon}$ belongs to $ V$. We also have, thanks to Theorem \ref{nulls},  that there exists  $g \in L^2(\omega \times (0,T-\epsilon))$ such that the associated solution $y$ to the problem
\begin{equation} \label{stokes-system-general-1}
\left |   
\begin{array}{ll}
y_t  - \Delta y + \nabla p = g\chi_{\omega}  &  \mbox{in}  \    (0,T-\epsilon)\times\Omega,  \\
div \ y = 0 & \mbox{in} \ (0,T-\epsilon)\times\Omega, \\
y = 0 & \mbox{on} \ (0,T-\epsilon)\times\partial \Omega, \\
y(0) = y_{\epsilon} & \mbox{in} \ \ \Omega,
\end{array}
\right. 
\end{equation}
satisfies
$$
y(T-\epsilon) = 0.
$$
Moreover,
\begin{equation}\label{controlestimate11}
\int_0^{T-\epsilon}\!\!\!\!\int_{\omega}|g|^2dxdt \leq C\gamma e^{\alpha L^2/T}||y_{\epsilon}||^2_V.
\end{equation}

Let us now define the  functions $\overline{y}$ and $\overline{g}$ by  $\overline{y}(t+\epsilon) = y(t)$,  $\overline{g}(t+\epsilon) = g(t) $ for $0<t<T-\epsilon$. The functions  $\overline{y}$ and $\overline{g}$ are defined in $(\epsilon, T)$ and satisfy

\begin{equation} \label{stokes-system-general-2}
\left |   
\begin{array}{ll}
\overline{y}_t  - \Delta \overline{y} + \nabla \overline{p} = \overline{g}\chi_{\omega}  &  \mbox{in}  \  (\epsilon,T)\times\Omega,  \\
div \ \overline{y} = 0 & \mbox{in} \  (\epsilon,T)\times\Omega, \\
\overline{y} = 0 & \mbox{on}  \ (\epsilon,T)\times\partial \Omega, \\
\overline{y}(\epsilon) = y_{\epsilon} & \mbox{in} \  \Omega.
\end{array}
\right. 
\end{equation}
Inequality  (\ref{controlestimate11}) then becomes
\begin{equation}\label{controlestimate2}
\int_{\epsilon}^{T}\!\!\!\!\int_{\omega}|\overline{g}|^2dxdt \leq C\gamma e^{\alpha L^2/T}||y_{\epsilon}||^2_V.
\end{equation}

Next, we set
$$
g(t) = \left \{ 
\begin{array}{rcc}
0; & \mbox{if} & 0 <t< \epsilon ,\\
\overline{g}(t); &\mbox{if}& \epsilon \leq t < T.
\end{array}
\right.
$$

It is not difficult to see that the solution $y$ of (\ref{stokes-system}), with $g$ as a control,  fulfils  $y(T) = 0$. From (\ref{controlestimate2}), and the definition of $g$, we have the following estimate
\begin{equation}\label{controlestimate3}
\int_{0}^{T}\!\!\!\!\int_{\omega}|g|^2dxdt \leq C\gamma e^{\alpha L^2/T}||y_{\epsilon}||^2_V.
\end{equation}



Let us now consider  system (\ref{stokes-system}) in the interval $[0,\epsilon]$, i.e., we consider the system

\begin{equation} \label{stokes-system-uncontrolled22}
\left |   
\begin{array}{ll}
y_t  - \Delta y + \nabla p = 0 &  \mbox{in}  \   (0,\epsilon)\times\Omega,  \\
div \ y  = 0 & \mbox{in} \  (0,\epsilon)\times\Omega, \\
y = 0 & \mbox{on} \  (0,\epsilon)\times\partial \Omega, \\
y(0) = y_0 & \mbox{in} \  \Omega,
\end{array}
\right. 
\end{equation}
with $y_0 \in H$.

We make  the change of variable $z(t) = e^{-\frac{1}{t}}y(t)$. This new function $z$  solves
\begin{equation} \label{stokes-system-uncontrolled33}
\left |   
\begin{array}{ll}
z_t  - \Delta z + \nabla p = \frac{1}{t^2}e^{-\frac{1}{t}}y &  \mbox{in}   \  (0,\epsilon)\times\Omega,  \\
div \ z = 0 & \mbox{in} \ (0,\epsilon)\times\Omega, \\
 z= 0 & \mbox{on} \  (0,\epsilon)\times \partial \Omega, \\
z(0) = 0 & \mbox{in} \  \Omega.
\end{array}
\right. 
\end{equation}

Using the fact that $\frac{1}{t^2}e^{-\frac{1}{t}}y \in L^2(0,\epsilon;H)$, and the regularity of the Stokes system,  we conclude  that $z \in L^2(0,\epsilon; H^2(\Omega))$ and that $z_t \in L^2(0,\epsilon;H)$.  

Multiplying (\ref{stokes-system-uncontrolled33}) by $z_t$  and integrating by parts, we get
\be\label{ener-1}
2|z_t(t)|^2_H + \frac{d}{dt} ||z(t)||^2_V = 2(\frac{1}{t^2}e^{-\frac{1}{t}}y(t),z_t)_H.
\ee
Integrating \eqref{ener-1}  from $0$ to $\epsilon$ and using Young's inequality, we obtain
$$
2\int_0^{\epsilon}|z_t(t)|^2_Hdt +  ||z(\epsilon)||^2_V \leq C_{\delta} \int_0^{\epsilon} |\frac{1}{t^2}e^{-\frac{1}{t}}y(t)|^2_Hdt    + \delta  \int_0^{\epsilon}|z_t|^2_Hdt,
$$
for any $\delta >0$.

Taking $\delta$ small enough, we have
\be\label{ener-2}
||z(\epsilon)||^2_V \leq  C\int_0^{\epsilon} |\frac{1}{t^2}e^{-\frac{1}{t}}y(t)|^2_Hdt
\ee
and since, for $\epsilon$ sufficiently small,  $\frac{1}{t^4}e^{-\frac{2}{t}} \leq e^{\frac{1}{\epsilon}}$, it follows that
$$
||z(\epsilon)||^2_V \leq   e^{\frac{1}{\epsilon}} \int_0^{\epsilon} |y(t)|^2_Hdt.
$$
Finally, using the fact that  $ ||y||^2_{L^2(0,\epsilon;H)} \leq \epsilon |y_0|^2_H$, we get from \eqref{ener-2} that
$$
||z(\epsilon)||^2_V \leq   \epsilon  e^{\frac{1}{\epsilon}} |y_0|^2_H,
$$
and, in particular, using the fact that $z(t) = e^{-\frac{1}{t}}y(t)$, we conclude  that
\begin{equation}\label{epsilonestimate}
||y(\epsilon)||^2_V \leq  \epsilon  e^{\frac{2}{\epsilon}} |y_0|^2_H.
\end{equation}

From \eqref{controlestimate3} and \eqref{epsilonestimate}, the result follows.
\end{proof}

\begin{remark}
Since $y_{\epsilon} \longrightarrow y_0$ in $H$, the norm of $y_{\epsilon}$ is not bounded in $V$. Hence, the right-hand side of  (\ref{controlestimate3}) is unbounded when $\epsilon \longrightarrow 0$. 
\end{remark}


\section{Null controllability for the hyperbolic system}\label{prooftheorem2.2}


This section is devoted to prove Theorem \ref{nullw} used in the proof of Theorem \ref{nulls}.  In order to prove the result, it is  convenient to write system \eqref{hyperpress}  in an abstract way. For that, we introduce  the Stokes operator $A: H^2(\Omega)^N \cap V \longrightarrow H$ given by
\be\label{stokesoperator}
Au := P(\Delta u),  
\ee
where $P: L^2(\Omega)^N \longrightarrow H$ is the orthogonal projection onto $H$ and $\Delta: H^2(\Omega)^N \cap H^1_0(\Omega)^N \longrightarrow L^2(\Omega)^N$ is the Laplace operator with Dirichlet boundary conditions. Thus, system \eqref{hyperpress}  is  equivalent to 
\begin{equation}\label{eq21}
\left |   
\begin{array}{ll}
u_{tt}  =  Au + h 1_{\omega}, \\
u(0) = u^0,  u_t(0) = u^1.
\end{array}
\right. 
\end{equation} 

The following theorem holds.
\begin{theorem}\label{existenciasol}
Let $(u^0,u^1, h) \in V \times H \times L^2(0,T;H )$. There exists a unique (weak) solution $ u$ of the problem (\ref{eq21}) such that
$$
u \in C([0,T]; V) \cap C^1([0,T]; H )
$$
and $u$ satisfies:
$$
\frac{1}{2}|u_t(t)|^2_H + \frac{1}{2}||u(t)||^2_V = \frac{1}{2}|u^1|^2_H + \frac{1}{2}||u^0||^2_V  + \int_0^t(h(s)1_{\omega},u_t(s))_{H}ds, \ \ \forall t \in [0,T].
$$

Moreover, the linear  mapping 
$$
V \times H \times L^2(0,T;H) \longrightarrow  C([0,T]; V) \cap C^1([0,T]; H )
$$
$$
(u^0,u^1, f) \mapsto u
$$
is continuous.
\end{theorem}

The  proof of Theorem \ref{existenciasol} is standard and, being far from the aim of this paper,   it will not be reproduced here (for a proof see, for instance, \cite{Tuc}).

\begin{remark}
Arguing as in chapter 2 of \cite{Tem} or in \cite{Simon}, it is possible to show the existence of a function $p \in H^{-1}(0,T;L^2_0(\Omega))$ such that  \eqref{hyperpress}  is satisfied in $\mathcal{D}'(Q)$. Moreover, there exists $C>0$ such that  
$$
||p||^2_{H^{-1}(0,T;L^2_0(\Omega))} \leq C(|u^1|^2_H + ||u^0||^2_V + ||h1_{\omega}||^2_{L^2(0,T;H)}).
$$
\end{remark}

By a classical duality argument (see, for instance, \cite{F-I1, Lions-livro}  ), it is not difficult to see  that proving Theorem \ref{nullw} is equivalent to show the existence of a positive  constant $C $ such that
\begin{equation} \label{desobs}
|\phi^0|_{H}^2 + ||\phi^1||^2_{V'} \leq C \int \! \! \! \int_{\omega\times (0,T)}|\phi|^2 dxdt,
\end{equation}
for all solutions of 
 \begin{equation}\label{eq22}
\left |   
\begin{array}{ll}
\phi_{tt}  = A\phi,   \\
\phi(0) = \phi^0,   \phi_t(0) = \phi^1,
\end{array}
\right. 
\end{equation} 
where $\phi^0 \in H$ and $\phi^1 \in V'$.

\begin{remark}\label{defultra}
Since the Stokes operator $A$ is an isomorphism from $V$ to $V'$, given  $(\phi^0,\phi^1) \in H\times V'$,  we define the solution $\phi$ of (\ref{eq22})  as 
$$
\phi =\psi_t,
$$
where $\psi$ is the unique solution of 
\begin{equation}\label{psi-trans}
\left |   
\begin{array}{ll}
\psi_{tt} = A\psi, \\
\psi(0) = A^{-1} \phi^{1},  \psi_t(0) = \phi^0.
\end{array}
\right. 
\end{equation} 
Following the arguments of \cite{Simon}, we can show that  for regular initial data the abstract problem  \eqref{eq22} is  equivalent to
\begin{equation}\label{hyperpress-1}
\left |   
\begin{array}{ll}
\phi_{tt}  - \Delta \phi + \nabla p = 0 &  \mbox{in}  \   Q,  \\
div \ \phi  = 0 & \mbox{in} \  Q, \\
\phi = 0 & \mbox{on} \  \Sigma, \\
\phi(0) = \phi^0,  \phi_t(0) = \phi^1 & \mbox{in} \  \Omega.
\end{array}
\right. 
\end{equation}
\end{remark}

Let us now  concentrate on proving \eqref{desobs}. The proof relies on some results that we prove below.
\begin{lemma}\label{equivalencia}
If, for every $(\phi^0,\phi^1) \in V \times H$,  the solution $\phi$ of (\ref{eq22}) satisfies 
 \begin{equation} \label{desobs2}
||\phi^0||_V^2 + |\phi^1|^2_{H} \leq C \int \! \! \! \int_{\omega \times (0,T)} |\phi_t|^2 dxdt,
\end{equation}
 for some  constant $C > 0 $, then inequality \eqref{desobs} holds for all  solutions of  \eqref{eq22} with initial data $(\phi^0,\phi^1)$ in $ H \times V'$.
 \end{lemma}

 \begin{proof}[Proof of Lemma \ref{equivalencia}]

Given $(\phi^0, \phi^1) \in H\times V' $, we consider $\psi$ solution of \eqref{psi-trans}, i.e., 
\begin{equation}\label{psi}
\left |   
\begin{array}{ll}
\psi_{tt}  = A\psi, \\
\psi(0) =  A^{-1} \phi^{1},  \psi_t(0) = \phi^0.
\end{array}
\right. 
\end{equation} 

Next, using the fact that $\phi =\psi_t$, and inequality  \eqref{desobs2},  we see that 
\begin{equation}\label{ob-1}
|| A^{-1} \phi^{1}||_V^2 + |\phi^0|^2_{H} \leq C\int \! \! \! \int_{\omega \times (0,T)}  |\phi|^2 dxdt.
\end{equation}
From \eqref{ob-1} and the fact that  $A: V\longrightarrow V'$ is an isomorphism, we finish the proof.
\end{proof}

\begin{lemma}\label{obsX}
Let $m \in C^1(\overline{\Omega})^N$. Then,  for all regular solutions of \eqref{eq22}, the following identity holds
\be
\left <  \nabla p, m \cdot \nabla \phi \right>_{L^2(Q)^N} =  - \left< \nabla p, \phi \cdot  \nabla m \right>_{L^2(Q)^N} + \left< \nabla p, \phi (div \ m) \right>_{L^2(Q)^N}.
\ee
\end{lemma}
\begin{proof}
Let us set  
$$
X = - \int \! \! \! \int_{Q} \frac{\partial p}{\partial x_i} m_k\frac{\partial \phi^i}{\partial x_k}dxdt.
$$
Integrating by parts with respect to $x_k$, and using the fact that $ \phi = 0 $ on $\Sigma$, we get
$$
X = \int \! \! \! \int_{Q}\frac{\partial^2p}{\partial x_k\partial x_i}m_k\phi^idxdt +\int \! \! \! \int_{Q} \frac{\partial p}{\partial x_i}\frac{\partial m_k}{\partial x_k}\phi^idxdt.
$$
Next, we integrate  by parts again the first integral, this time with respect to $x_i$, and we obtain
$$
\int \! \! \! \int_{Q} \frac{\partial p}{\partial x_k}\frac{\partial}{\partial x_i}\biggl(m_k\phi^i\biggl)dxdt = -\int \! \! \! \int_{Q} \nabla p \phi \cdot  \nabla m dxdt.
$$  
Hence, we conclude  that
$$
X = - \int \! \! \! \int_{Q} \nabla p \phi \cdot  \nabla m dxdt +\int \! \! \! \int_{Q} \nabla p \phi (div \ m) dxdt,
$$
and the proof is finished.
\end{proof}

\begin{lemma}\label{equivalencia-X}
Assume  $\omega$ satisfies \eqref{hipm} and let $T > 2R_0$. Then there exists $C>0$ such that,  for every $(\phi^0,\phi^1) \in V \times H$,  the weak solution $\phi$ of (\ref{eq22}) satisfies: 
\be\label{firstobs}
||\phi^0||_V^2 + |\phi^1|^2_{H} \leq C\int \! \! \! \int_{\omega \times (0,T)}\bigl(|\phi_t|^2+|\phi|^2\bigl)dxdt.
\ee
 \end{lemma}

\begin{proof}
Along the proof we use the following notation:
$$
E(t) = |\phi_t(t)|_H^2 + ||\phi(t)||_V^2, \ \ \forall t \in [0,T].
$$ 
Without loss of generality, we assume that $\phi$ is regular and work with the equivalent problem  \eqref{hyperpress-1}, this is the case if we  take, for instance,  $\phi^0 \in V\cap H^4(\Omega)$ and $\phi^1 \in V\cap H^2(\Omega)$. 

Using the change of variables $T\tau = (T-2\epsilon)t + T\epsilon$, which implies $\epsilon \leq \tau \leq T-\epsilon $, from the boundary observability  inequality given in  Theorem \ref{lemma22} in the appendix,  we have
$$
E(0) \leq C\int_{\epsilon}^{T-\epsilon} \! \!\!\! \int_{\partial \Omega} \biggl(   \frac{\partial \phi}{\partial \nu}\biggl)^2d\Sigma.
$$

Next, we consider a vector field $h \in C^2(\overline{\Omega})^N$ such that  $h = \nu $ on $\partial \Omega$ and $h = 0$ on $\Omega \setminus \omega$ and let $\eta \in C^2([0,T])$ be such that $\eta(0) = \eta(T) = 0$ and $ \eta(t) = 1 $ in $(\epsilon, T-\epsilon)$.  We define $\theta(x,t) = \eta(t)h(x)$, which belongs to $W^{2,\infty}(Q)$ and satisfies
\begin{equation*}
\left |   
\begin{array}{ll}
\theta(x,t) = \nu(x)   \ \  \mbox{for all}  \  \  (x,t) \in  (\epsilon,T-\epsilon) \times \partial \Omega,  \\
\theta(x,t)\cdot \nu(x)  \geq 0, \ \, \mbox{for all} \ \ (x,t) \in  (0,T) \times \partial \Omega, \\
\theta(x,0) = \theta(x,T) = 0 , \ \ \mbox{for all} \ \ x \in  \Omega, \\
\theta(x,t) = 0 \  \ \mbox{in} \ \ \bigl(\Omega \setminus \omega\bigl) \times (0,T).
\end{array}
\right. 
\end{equation*}

Then we consider the multiplier $\theta \cdot  \nabla \phi$ and, from Lemma \ref{lema21} in the appendix, we  obtain the following identity for all weak solution $\phi $ of (\ref{eq22}):
\begin{eqnarray}
\frac{1}{2}\int \! \! \! \int_{\Sigma}\theta_k(x,t) \nu_k(x)\biggl(\frac{\partial \phi}{\partial \nu}\biggl)^2d\Sigma &=&  (\phi_t(.), \theta(x,.)\cdot \nabla \phi(.))\bigl|_0^T +\int \! \! \! \int_{Q}\frac{\partial \theta_k}{\partial x_j}\frac{\partial \phi^i}{\partial x_k}\frac{\partial \phi^i}{\partial x_j}dxdt  \\
&& + \frac{1}{2} \int \! \! \! \int_{Q}\frac{\partial \theta_k}{\partial x_k} \bigl(  |\phi_t|^2 - |\nabla \phi|^2 \bigl)dxdt +\int \! \! \! \int_{Q}\frac{\partial p}{\partial x_i} \theta_k\frac{\partial \phi^i}{\partial x_k}dxdt \nonumber .
\end{eqnarray}

Using the definition of  $\theta$, we have
$$
   \frac{1}{2} \int_{\epsilon}^{T- \epsilon} \!\!\!\!\int_{\partial \Omega}\biggl(\frac{\partial \phi}{\partial \nu}\biggl)^2d\Sigma \leq \frac{1}{2}\int \! \! \! \int_{\Sigma}\theta_k(x,t) \nu_k(x)\biggl(\frac{\partial \phi}{\partial \nu}\biggl)^2d\Sigma,
$$
because $\theta(x,t) = \nu(x) $ on $\partial \Omega \times (\epsilon, T-\epsilon)$ and
$$
(\phi_t(.), \theta(x,.)\cdot \nabla \phi(.))\bigl|_0^T  = 0.
$$

 We also have
$$
\biggl| \int \! \! \! \int_{Q}\frac{\partial \theta_k}{\partial x_j}\frac{\partial \phi^i}{\partial x_k}\frac{\partial \phi^i}{\partial x_j}dxdt \biggl| \leq  C\int \! \! \! \int_{\omega \times (0,T)}|\nabla \phi|^2 dxdt,
$$
since $ \theta \in C^1(\overline{\Omega}\times (0,T))$.

For the pressure, we use Lemma \ref{obsX} to see that 
\begin{align}
\int \! \! \! \int_{Q}\frac{\partial p}{\partial x_i} \theta_k\frac{\partial \phi^i}{\partial x_k}dxdt  
 = \left < \nabla p, -\phi\cdot  \nabla \theta +  \phi (div \ \theta)   \right >_{{H^{-1}(Q)}^N, {H^{1}_0(Q)^N}}. \nonumber
\end{align}
Consequently 
\begin{align}
\left | \int \! \! \! \int_{Q}\frac{\partial p}{\partial x_i} \theta_k\frac{\partial \phi^i}{\partial x_k}dxdt \right|  \leq & \  C_\delta \int \! \! \! \int_{\omega \times (0,T)} (  |\phi|^2+ |\phi_t|^2 + |\nabla \phi|^2) dxdt \nonumber \\
&+ \delta \left\| \nabla p \right \|^2_{H^{-1}(Q)^N},
\end{align}
for any $\delta >0$.  Thus, 
$$
 \frac{1}{2}\int \! \! \! \int_{\Sigma}\theta_k(x,t) \nu_k(x)\biggl(\frac{\partial \phi}{\partial \nu}\biggl)^2d\Sigma  \leq  C \int \! \! \! \int_{\omega \times (0,T)}\bigl(|\phi|^2+ |\phi_t|^2 + |\nabla \phi|^2 \bigl)dxdt+\delta \left\| \nabla p \right \|^2_{H^{-1}(Q)^N}.
 $$
Using the fact that 
$$
\left\| \nabla p \right \|^2_{H^{-1}(Q)^N} \leq C E(0), 
$$
and  choosing $\delta$ small enough,  we conclude that 
\be\label{FGB}
E(0) \leq C \int_{\epsilon}^{T- \epsilon}\!\!\!\!\!\int_{\partial \Omega}\biggl(\frac{\partial \phi}{\partial \nu}\biggl)^2d\Sigma  \leq C \int \! \! \! \int_{\omega \times (0,T)}\bigl( |\phi_t|^2 + |\phi|^2 + |\nabla \phi|^2 \bigl)dxdt.
\ee
Hence, by change of variables, we have  that 
 \begin{equation}\label{id21}
E(0)   \leq C \int_{\epsilon}^{T-\epsilon}\!\!\!\!\int_{{\omega}}\bigl( |\phi|^2+ |\phi_t|^2 + |\nabla \phi|^2 \bigl)dxdt.
\end{equation}

Now, let $\omega_0$ be a neighborhood of $\partial \Omega$ such that $\Omega\cap\omega_0 \subset \omega$.  We observe that  inequality (\ref{id21}) is true for each neighborhood of $\partial \Omega$, and in particular for $\omega_0$, that is to say
$$
E(0)   \leq C \int_{\epsilon}^{T-\epsilon}\!\!\!\!\!\int_{{\omega_0}}\bigl(  |\phi|^2 + |\phi_t|^2 + |\nabla \phi|^2 \bigl)dxdt.
$$

Now, we consider $\rho \in W^{1,\infty}(\Omega)$, $\rho \geq 0$, such that
$$
\rho = 1 \ \ \mbox{in} \ \  \omega_0, \ \ \mbox{and} \ \ \rho = 0 \ \ \mbox{in} \ \ \Omega \setminus \omega. 
$$

Defining $h = h(x,t)$ by $h(x,t) = \eta(t)\rho^2(x)$, where $\eta$ is defined above, it follows that

\begin{equation*}
\left |   
\begin{array}{ll}
h(x,t) = 1   \ \  \mbox{for all}  \  \  (x,t) \in \omega_0 \times (\epsilon,T-\epsilon),  \\
h(x,t)  = 0, \ \, \mbox{for all} \ \ (x,t) \in \bigl(\Omega \setminus \omega\bigl) \times (0,T), \\
h(x,0) = h(x,T) = 0 , \ \ \mbox{for all} \ \ x \in  \Omega, \\
\frac{|\nabla h|}{h} \in L^{\infty}(Q).
\end{array}
\right. 
\end{equation*}

Multiplying both sides of $ \eqref{hyperpress-1}_1$ by $h\phi$ and integrating by parts in $Q$, we obtain 
$$
\int \! \! \! \int_{Q} h \phi\cdot \phi_{tt}dxdt -\int \! \! \! \int_{Q} h\phi \cdot \Delta \phi dxdt + \int \! \! \! \int_{Q} h \nabla p \cdot \phi dxdt = 0.
$$ 
We have
\be\label{VFBG}
\int \! \! \! \int_{Q} h \phi_{tt} \cdot \phi dxdt = -\int \! \! \! \int_{Q} h|\phi_t|^2 dx dt -\int \! \! \! \int_{Q} h_t\phi \cdot \phi_t dxdt.
\ee
For the second term in the right hand side of \eqref{VFBG}, since  $\phi = 0$ on $\Sigma$, we have 
$$
-\int \! \! \! \int_{Q} h\Delta \phi \cdot \phi dxdt   =\int \! \! \! \int_{Q} h |\nabla \phi|^2 dxdt + \int \! \! \! \int_{Q} \phi\cdot\bigl(\nabla \phi \cdot \nabla h \bigl)dxdt.
$$
Consequently,
$$
\int \! \! \! \int_{Q} h|\nabla \phi |^2dxdt =  \int \! \! \! \int_{Q} h|\phi_t|^2 dxdt+\int \! \! \! \int_{Q}h_t\phi\cdot\phi_t dxdt -\int \! \! \! \int_{Q}  \phi\cdot\bigl(\nabla \phi \cdot \nabla h \bigl)dxdt  - \int \! \! \! \int_{Q} h \nabla p \cdot \phi dxdt.
$$
It is immediate that
$$
\biggl| \int \! \! \! \int_{Q} \phi \cdot \bigl( \nabla \phi \cdot \nabla h \bigl)dxdt \biggl| \leq \frac{1}{2}\int \! \! \! \int_{Q} h|\nabla \phi|^2 dxdt + \frac{1}{2} \int \! \! \! \int_{Q} \frac{|\nabla h|^2}{h}|\phi|^2 dxdt.
$$
Hence
$$
\int \! \! \! \int_{Q} h|\nabla \phi |^2dxdt \leq C \int \! \! \! \int_{\omega \times (0,T)}\bigl(|\phi_t|^2+|\phi|^2\bigl)dxdt+
2 \biggl| \int \! \! \! \int_{Q} h \nabla p \cdot \phi dxdt\biggl|. 
$$
Next, observing that 
\begin{align}
  \int \! \! \! \int_{Q} h \nabla p \cdot \phi dxdt  
   \leq \delta || p||^2_{H^{-1}(0,T; L^2(\Omega)^N)}+  C_\delta || h \phi ||^2_{H^{1}_0(0,T; L^2(\Omega)^N)}, \nonumber 
 \end{align}
for any $\delta >0$, we conclude that  
$$
\int_{\epsilon}^{T-\epsilon}\!\!\!\!\!\int_{\omega_0}|\nabla \phi|^2dxdt \leq C\int \! \! \! \int_{\omega \times (0,T)}\bigl(|\phi_t|^2+|\phi|^2\bigl)dxdt +\delta || p||^2_{H^{-1}(0,T; L^2(\Omega)^N)}.
$$
From this last estimate we infer that
$$
E(0) \leq C\int \! \! \! \int_{\omega \times (0,T)}\bigl(|\phi_t|^2+|\phi|^2\bigl)dxdt +\delta || p||^2_{H^{-1}(0,T; L^2(\Omega)^N)}.
$$
Finally, taking  $\delta$ small enough, we obtain
\begin{align}\label{almost}
E(0) \leq C\int \! \! \! \int_{\omega \times (0,T)}\bigl(|\phi_t|^2+|\phi|^2\bigl)dxdt,
\end{align}
which is exactly \eqref{firstobs}.
\end{proof}


\begin{proposition}\label{equivalencia-1}
Assume  $\omega$ satisfies \eqref{hipm}. Then there exist $T_0 > 0$ and a constant $C > 0 $ such that,  for any  $T >  T_0$ and any $(\phi^0,\phi^1) \in V \times H$,  the solution $\phi$ of (\ref{eq22}) satisfies  \eqref{desobs2}. 
 \end{proposition}

\begin{proof}[Proof of Proposition \ref{equivalencia-1} ]
Let us suppose that (\ref{desobs2}) is not true. Then, given a natural number $n$, there exists an initial data $(\tilde{\phi^0_n},\tilde{\phi^1_n})$ such that  $\tilde{\phi_n}$, the solution  of  (\ref{eq22}) corresponding to this initial data,  satisfies
$$
||\tilde{\phi^0_n}||^2_{V}+ |\tilde{\phi^1_n}|^2_H \geq n ||\tilde{\phi}_{n,t}||_{L^2(\omega \times (0,T))}.
$$
Without loss of generality, we assume  that $(\tilde{\phi^0_n},\tilde{\phi^1_n})$ is smooth and  set
$$
K = \biggl(   ||\tilde{\phi^0_n}||^2_{V}+ |\tilde{\phi^1_n}|^2_H \biggl)^{1/2}
$$
and
$$
\phi^0_n = \frac{\tilde{\phi^0_n}}{K}, \ \   \phi^1_n = \frac{\tilde{\phi^1_n}}{K}, \ \ \phi_n = \frac{\tilde{\phi_n}}{K}.
$$

We have
\begin{align}\label{36}
||\phi_{n,t}||^2_{L^2(\omega \times (0,T))} \leq \frac{1}{n}
\end{align}
and
\be\label{bi}
||\phi^0_n||^2_{V}+ |\phi^1_n|^2_H = 1.
\ee

From \eqref{36},  there exist subsequences, denoted by the same index,  such that
\begin{equation}\label{123}
\liminf_{n \longrightarrow \infty}\int \! \! \! \int_{\omega \times (0,T)}  |\phi_{n,t}|^2dxdt = 0,
\end{equation}
\be
\phi_n^0 \rightharpoonup \phi^0 \ \ \mbox{in} \ \ V
\ee
and 
\be
\phi_n^1 \rightharpoonup \phi^1 \ \ \mbox{in} \ \ H.
\ee

Since $\phi_n$ is the solution of (\ref{eq22}) associated to the initial data $(\phi_n^0,\phi_n^1)$,  we have:
\begin{equation}\label{38}
\left |   
\begin{array}{ll}
\phi_n \ \ \mbox{is bounded in} \ \ L^{\infty}(0,T; V), \\
\phi_{n,t} \ \ \mbox{is bounded in} \ \ L^{\infty}(0,T; H).
\end{array}
\right.
\end{equation}
Therefore, there exists a subsequence $\{\phi_n\}$ such that
\begin{equation}\label{38}
\left |   
\begin{array}{ll}
\phi_n \longrightarrow \phi \ \ \mbox{weak star in} \ \ L^{\infty}(0,T; V), \\
\phi_{n,t} \longrightarrow \phi_t  \ \ \mbox{weak star in} \ \ L^{\infty}(0,T; H).
\end{array}
\right.
\end{equation}

From \eqref{38}, it is not difficult to show that  $\phi$ is the weak solution of (\ref{eq22}) corresponding to the  initial data $(\phi^0,\phi^1)$.

Next, since $V \hookrightarrow H$ compactly,  estimate (\ref{38}) and the Aubin-Lions compactness theorem give 
\begin{equation} \label{strongconv}
\phi_n \longrightarrow \phi \ \mbox{in } \ \ L^2(0,T; H).
\end{equation}

Hence, it follows from \eqref{123} and \eqref{38} that 
\be\label{hh}
\phi_t \equiv 0 \ \ \mbox{in} \ \ \omega \times (0,T)
\ee
and $\phi$ is independent of $t$ in $\omega$.

%

Let us now  consider  the system

 \begin{equation}\label{eq45}
\left |   
\begin{array}{ll}
\xi_{tt}  = A\xi,  \\
\xi(0) = \phi^1,  \xi_t(0) = A \phi^0.
\end{array}
\right. 
\end{equation} 

Taking $ \psi(x,t) = \phi^0(x) +\int_0^t\xi(x,s)ds$,  it is not difficult to see that  $\psi$ solves (\ref{eq22}), with $( \phi^0, \phi^1)$ as initial data. Therefore,  from the uniqueness  of solutions of (\ref{eq22}), we have that  $\psi \equiv \phi$ and thanks to \eqref{hh} we have that $\xi \equiv 0$ in $\omega \times (0,T)$.

Let us now show that $\xi \equiv 0$. Indeed, applying the $curl$ operator in (\ref{eq45}), we see that $v = curl \ \xi$ satisfies

\begin{equation}\label{eq46}
\left |   
\begin{array}{ll}
v_{tt}  - \Delta v = 0 &  \mbox{in}   \  Q,  \\
v \equiv 0 & \mbox{in}  \ \omega \times (0,T).
\end{array}
\right. 
\end{equation} 

Then, by Holmgren's Uniqueness Theorem (see \cite{Lions-livro}),  there exists $T_0> 0$ such that if $T > T_0$ then  $ v \equiv  0 $. Therefore, there exists a scalar function $\Phi = \Phi(x,t)$ such that
$$
\xi = \nabla \Phi \ \ \mbox{in} \ \ Q.
$$

In view of $(\ref{eq45})_{2}$,  we have
$$
\Delta \Phi = 0 \ \ \mbox{in} \ \ Q.
$$

Since $\xi = 0$ in $\omega \times (0,T)$, we also have that
$$
\Phi = f(t) \ \ \mbox{in} \ \ \omega \times (0,T).
$$
From the unique continuation for the Laplace equation, we deduce that
$$
\Phi = f(t) \ \ \mbox{in} \ \ Q,
$$
which implies 
\be\label{ci}
 \xi = \nabla \Phi = 0 \ \ \mbox{in} \ \ Q.
 \ee
Hence,
 \be\label{pi}
 \phi^1= \phi^0= 0.
\ee
From \eqref{firstobs}, \eqref{strongconv}  and \eqref{pi},  we get a contradiction, and the proof is finished.
\end{proof}

As a consequence of Lemmas \ref{equivalencia} and  \ref{equivalencia-X},  and Proposition \ref{equivalencia-1}, we have the following result. 

\begin{theorem}\label{controlwavepre}
Assume  $\omega$ satisfies \eqref{hipm}. Then there exist $T_0 > 0$ and a constant $C > 0 $ such that  for any  $T >  T_0$ and any $(\phi^0,\phi^1) \in H \times V'$,  the solution $\phi$ of (\ref{eq22}) satisfies  \eqref{desobs}. 
\end{theorem}

We end this section proving  Theorem \ref{nullw}.

\begin{proof}[Proof of Theorem \ref{nullw}] 

We consider the functional
\begin{equation}\label{funcionalJ}
\mathcal{J} : H \times V' \longrightarrow \mathbb{R}
\end{equation}
given by 
\begin{equation*}
\mathcal{J}(\phi^0, \phi^1) = \frac{1}{2}\int \! \! \! \int_{\omega \times (0,T)} |\phi|^2dxdt + <\varphi^1,u^0>_{V,V'} - (\phi^0, u^1)_H, 
\end{equation*}
where $\varphi$ is the solution of (\ref{eq22}) corresponding to the initial data $(\phi^0,\phi^1)$.

Using the observability inequality \eqref{desobs} and energy estimates, we can show that the functional  $\mathcal{J}$ is continuous, strictly convex and  coercive. Therefore,   $\mathcal{J}$ has a unique minimizer $(\hat{\phi}^0,\hat{\phi}^1)$. Using the Euler-Lagrange equation of $ \mathcal{J}$,  we conclude that $\hat{\phi}$,  solution of (\ref{eq22})  associated to $(\hat{\phi}^0,\hat{\phi}^1)$,  is a control which drives $u$ to zero at time $T$.
Inequality \eqref{costw} then  follows from the observability inequality \eqref{desobs} and the fact that $\mathcal{J}(\hat{\phi}^0,\hat{\phi}^1) \leq 0$. This finishes the proof of Theorem \ref{nullw}.

\end{proof}

\begin{remark}
The minimal time $T_0$ in Proposition \ref{equivalencia-1}  and Theorems  \ref{nullw} and  \ref{controlwavepre} must  satisfy $T_0 > 2R_0$  and be  such that  Holgrem's Theorem can be applied to conclude that the solution of \eqref{eq46} is zero (see \cite{Lions-livro}).
\end{remark}

\appendix

\section*{Acknowledgements}
The author thanks D. A. Souza,  J.-P. Puel  and   E. Zuazua  for valuable discussions and comments  related to this paper. This work was partially supported by the Grant  BFI-2011-424 of the Basque Government and partially supported by the Grant  MTM2011-29306-C02-00 of the MICINN, Spain, the ERC Advanced Grant FP7-246775 NUMERIWAVES, ESF Research Networking Programme OPTPDE and the Grant PI2010-04 of the Basque Government.

\section{Boundary observability  for the hyperbolic system}\label{ap}

This section is devoted  to prove the following result.

\begin{theorem}\label{lemma22}
If we take  $T > 2R_0$ then, for every solution of (\ref{eq22}) with initial data  $(\phi^0,\phi^1) \in V \times H$, the following estimate holds:
\begin{equation}\label{obsfronteira}
|\phi^1|_H^2 + ||\phi^0||_V^2 \leq \frac{R_0}{2(T-2R_0)} \int\!\!\!\!\int_{\Sigma}\biggl(   \frac{\partial \phi}{\partial \nu}\biggl)^2d\Sigma.
\end{equation}
\end{theorem}

For the  proof of  Theorem \ref{lemma22}, we need the following two lemmas.

\begin{lemma}\label{lema21}
Let $\overline{q} = \overline{q}(x) $ be in $ C^1(\bar{\Omega})^N$, then, for every regular solution  $u$  of (\ref{eq21}), the following identity holds:

\begin{eqnarray}
\frac{1}{2}\int \! \! \! \int_{\Sigma}\overline{q}_k(x) \nu_k(x)\biggl(\frac{\partial u}{\partial \nu}\biggl)^2d\Sigma &=&  (u_t(t), \overline{q}(x)\nabla u(t))\bigl|_0^T +\int \! \! \! \int_{Q}\frac{\partial \overline{q}_k}{\partial x_j}\frac{\partial u^i}{\partial x_k}\frac{\partial u^i}{\partial x_j}dxdt \nonumber  \\
&& + \frac{1}{2}  \int \! \! \! \int_{Q}\frac{\partial \overline{q}_k}{\partial x_k} \bigl(  |u_t|^2 - |\nabla u|^2 \bigl)dxdt  \nonumber \\
&& +\int \! \! \! \int_{Q}\frac{\partial p}{\partial x_i} \overline{q}_k\frac{\partial u^i}{\partial x_k}dxdt  +\int \! \! \! \int_{Q} h^i \overline{q}_k\frac{\partial u^i}{\partial x_k}dxdt.
\end{eqnarray}

\end{lemma}

The proof of Lemma \ref{lema21} is the same as in the case of a single wave equation, the difference being that here we see the pressure as a force term in the right-hand side.

%

%

\begin{lemma} \label{desigualdadedireta}
Let $(u^0, u^1, h) \in V\times H \times L^2(Q)^N$, then the weak solution of (\ref{eq21}) satisfies:
$$
\int \! \! \! \int_{\Sigma} \biggl(  \frac{\partial u}{\partial \nu}\biggl)^2 d\Sigma \leq C\bigl(|u^1|^2_H + ||u^0||^2_V + ||h||^2_{L^2(Q)^N}\bigl).
$$
\end{lemma}

\begin{proof}
The proof is obtained exactly as in the case of  the wave equation,  first showing the result for regular solutions. Indeed, in this case we must  take the vector field $\overline{q}$ in Lemma \ref{lema21} to be the vector field $\overline{q}(x) = x$ and use the fact that 
$$\int \! \! \! \int_{Q}\frac{\partial p}{\partial x_i} \overline{q}_k\frac{\partial u^i}{\partial x_k}dxdt  = 0.$$ 
\end{proof}

\begin{proof}[Proof of Lemma \ref{lemma22}]
Without loss of generality, we assume that $\phi$ is regular and then work with the equivalent problem   \eqref{hyperpress-1}. Using Lemma \ref{lema21}, with $\overline{q}$ being  the vector field $m(x) = x$, we have
$$
\frac{1}{2}\int \! \! \! \int_{\Sigma}m\cdot\nu\biggl(\frac{\partial \phi}{\partial \nu}\biggl)^2d\Sigma =  (\phi_t(.), m(x) \nabla \phi(.))\bigl|_0^T + \int \! \! \! \int_{Q}|\nabla \phi |^2dxdt  + \frac{N}{2} \int \! \! \! \int_{Q}\bigl(  |\phi_t|^2 - |\nabla \phi|^2 \bigl)dxdt.
$$
Next, multiplying  $ \eqref{hyperpress-1}_1$ by $\phi$ and integrating by parts,  we easily see that
$$
(\phi_t(.),\phi(.))\bigl|_0^T =  \int \! \! \! \int_{Q} |\phi_t|^2dxdt -  \int \! \! \! \int_{Q} |\nabla \phi|^2dxdt.
$$
Then, using this last identity and the fact that 
$$
  |\phi_t(t)|_H^2 + ||\phi(t)||_V^2  =  |\phi^1|_H^2 + ||\phi^0||_V^2 \ \ \ \forall t \in [0,T],
$$
we obtain
$$
(\phi_t(.),m\nabla u(.) +\frac{N-1}{2}u(.))\bigl|_0^T + T \bigl( |\phi^1|_H^2 + ||\phi^0||_V^2 \bigl) = \frac{1}{2} \int \! \! \! \int_{\Sigma}m\cdot\nu\biggl(\frac{\partial \phi}{\partial \nu}\biggl)^2d\Sigma.
$$
We also have
$$
\bigl|   m\nabla u(t) +\frac{N-1}{2}u(t) \bigl|^2 \leq R_0 |\nabla \phi(t) |^2 \  \ \ \forall t  \in [0,T],
$$
which implies, by Gronwall inequality, that
$$
\biggl|  \bigl( \phi_t(.),m \nabla \phi(.) +\frac{N-1}{2} \phi(.) ) \bigl|_0^T \biggl| \leq  2R_0 \bigl(|\phi^1|_H^2 + ||\phi^0||_V^2\bigl).
$$
Finally, combining all the above estimates, we conclude that
$$
\bigl(T-2R_0\bigl)\bigl(|\phi^1|_H^2 + ||\phi^0||_V^2\bigl) \leq \frac{R_0}{2} \ \!\!\!\!\int_{\Sigma}\biggl(\frac{\partial \phi}{\partial \nu}\biggl)^2d\Sigma,
$$
which is exactly \eqref{obsfronteira}.

\end{proof}

\end{document}